\documentclass[12pt]{amsart}
\usepackage{amsthm}
\usepackage{mathtools,amssymb,latexsym,graphics,enumerate}
\usepackage[mathscr]{eucal}
\usepackage{amsmath,amsfonts,amsthm,amssymb}
\usepackage{color}
\usepackage{xcolor}
\usepackage{hyperref}
\usepackage{caption}
\usepackage{subcaption}
\usepackage{cleveref}
\usepackage[left=1in,right=1in,top=1in,bottom=1in]{geometry}

\usepackage{graphicx}
\numberwithin{equation}{section}

\newcommand{\be}{\begin{eqnarray*}}
\newcommand{\bel}{\begin{eqnarray}}
\newcommand{\ee}{\end{eqnarray*}}
\newcommand{\eel}{\end{eqnarray}}
\newcommand{\ba}{\begin{aligned}}
\newcommand{\ea}{\end{aligned}}

\newtheorem{theorem}{Theorem}

\title[
A possible model for the origin of chemotaxis
]
{
A simple reaction-diffusion system as a possible model for the origin of chemotaxis
}
\author{Yishu Gong$^*$} \thanks{$^*$yishu.gong@duke.edu, Department of Mathematics, Duke University, Durham, NC, 27708, USA}
\author{Alexander Kiselev$^\dag$}\thanks{$^\dag$kiselev@math.duke.edu, Department of Mathematics, Duke University, Durham, NC, 27708, USA}
\date{\today}

\begin{document}

\begin{abstract}
Chemotaxis is a directed cell movement in response to external chemical stimuli. In this paper, we propose a simple model for the origin of chemotaxis - namely how a directed movement in response to an external chemical signal may occur based on purely reaction-diffusion equations reflecting inner working of the cells.
The model is inspired by the well studied role of the rho-GTPase Cdc42 regulator of cell polarity, in particular in yeast cells.
We analyze several versions of the model in order to better understand its analytic properties, and prove global regularity
in one and two dimensions. Using computer simulations, we demonstrate that in the framework of this model, at least in certain parameter regimes, the speed of the directed movement appears to be proportional to the size of the gradient of signalling chemical. This coincides with the form of the chemical drift in the most studied mean field model of chemotaxis, the Keller-Segel equation.
\end{abstract}

\maketitle

\section{Introduction}

Chemotaxis is a directed cell movement in response to external chemical stimuli.
Chemotaxis is ubiquitous in biology; for example, it plays a role in organism morphology \cite{solnica2012gastrulation,shellard2016chemotaxis}, reproduction processes \cite{himes2011sperm,ralt1994chemotaxis,riffell2007sex,zimmer2011sperm} and workings of immune system \cite{deshmane2009monocyte,taub1995monocyte}.
There are many mathematical models of chemotaxis; the most studied is the Keller-Segel equation and its variants.
Virtually all of these models incorporate a transport term driven by the concentration of the external chemical, which may be produced by the cells themselves (see e.g. \cite{keller1971model,perthame2006transport}, where further references can be found).
Yet we are not aware of any mathematical models that would aim to explain the origin of the transport based on reaction-diffusion processes taking place inside cells.
The way chemotaxis happens, at least for eukaryotic cells, is that cells translate chemical  environmental cues into amplified intracellular signaling that results in elongated cell shape, actin polymerization toward the leading edge, and movement along the gradient.
In this paper, instead of presenting chemotaxis as an explicit transport term, we explore model that aims to explain the origin the chemotactic ability of cells. Inspired by \cite{chiou2018principles}, we look at sexual yeast reproduction and simplify the polarization
process into understanding active rho-GTPase Cdc42 concentration in one yeast under chemical gradient produced by another yeast. This is certainly just an element of a more complex picture involved into producing chemotactic response in cells,
but we limit consideration to this one stage. Our first goal is to explore the well-posedness properties of the model and its variants and to understand analytic features involved. Our second goal is to get more information on the nature of transport
generated by the model reacting to external chemical stimuli. In particular a natural question is how the speed of transport, which we measure via the coordinate moment of a density, depends on the gradient of the attractive chemical.
This question we approach through numerical simulations, and find that for certain reasonable ranges of parameters, this dependence is linear.


The two-species mass-conserved activator-substrate (MCAS) system that is the basis of our model consists of two partial differential equations (PDEs) governing the kinetics of the slowly diffusing activator $u$ (GTP-bound GTPase on the membrane) and the rapidly diffusing substrate $v$ (GDP-bound GTPase in the cytoplasm). In general, this system has the following form in 1D (see \cite{chiou2018principles}):
\begin{align}\label{eqn:general_sys}
\begin{split}
    \frac{\partial u}{\partial t}&=F(u,v)+ k\Delta u, \\
    \frac{\partial v}{\partial t}&=-F(u,v) + k_v\Delta v.
\end{split}
\end{align}
Here, $k$ refers to the diffusion of $u$, $k_v$ refers to the diffusion of $v$. These two diffusion constants usually differ by two orders of magnitude. $F(u, v)$ describes the biochemical interconversions between $u$ and $v$, given in the form:
\begin{equation*}
    F(u,v)= h(u)v - g(u)u.
\end{equation*}
\begin{figure}[h]
\centering
\includegraphics[scale=0.35]{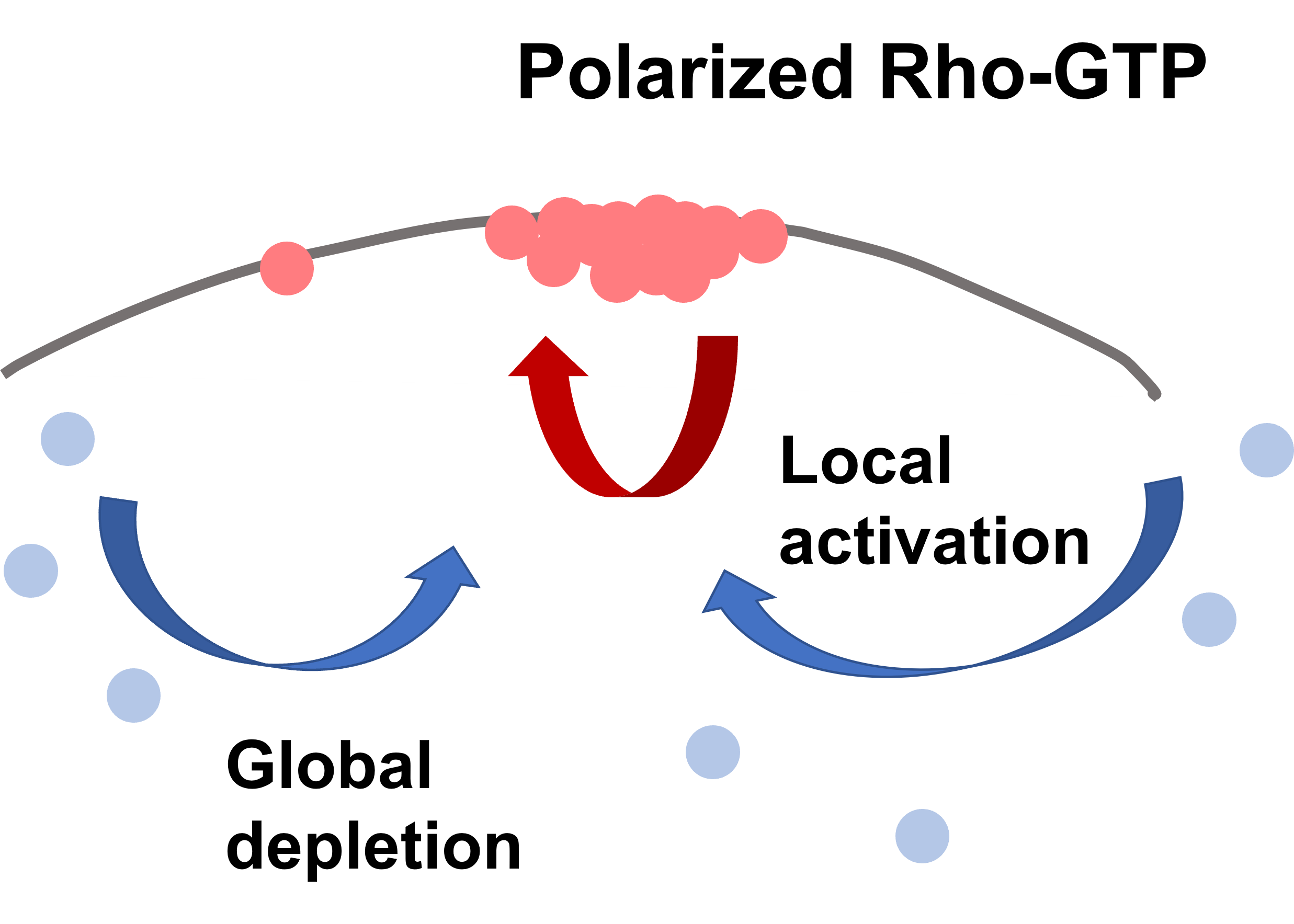}
\caption{Local activation via positive feedback and depletion of the substrate in
the cytosol generates an activator-enriched domain on the cortex \cite{chiou2018principles}.}\label{Fig:Problem_setup}
\end{figure}

The positive feedback (i.e. conversion from inactive GDPase to active GTPase with energy) is denoted by $f(u)v$, while the negative feedback (i.e. conversion from active GTPase to inactive GDPase without energy) is denoated by $g(u)u$. Examples of $F(u,v)$ includes:
the simplest
\begin{equation*}
    h(u) = \tilde{a} u^2, g(u) = b;
\end{equation*}
Goryachev’s (see \cite{goryachev2008dynamics})
\begin{equation}\label{eqn:goryachev}
    h(u) = \tilde{a}u^2 + cu, g(u) = b;
\end{equation}
and (see \cite{otsuji2007mass})
\begin{equation*}
    h(u) =1, g(u) = \frac{b}{1+u^2}.
\end{equation*}
In \cite{chiou2018principles}, Turing-type instability for these types of reaction have been analyzed; it was also shown
that steady states with more than one peak are unstable for many kinds of $F(u,v)$. This analysis is in agreement with experiments as one usually only observe one bud in yeast asexual production \cite{chiou2018principles}.

Several complicated computational models have been developed to mimic gradient-induced polarization toward the pheromone source \cite{wang2019mating, ismael2016gbeta} and have shown the rate of movement is dependent on pheromone concentration \cite{bendezu2013cdc42}.
Here, we propose a simpler system that is capable of capturing such gradient tracking ability, specifically, in the context of chemotactic reaction of a single yeast cell to an external pheromone signal.
We apply a modification to the Turing-type model described above and add a pheromone density profile $\tilde{\alpha}f(x)$ that depends on location in the form similar to \cite{wang2019mating} - and we obtain the following system:
\begin{align}\label{eqn:original_sys}
\begin{split}
    \frac{\partial u}{\partial t}&=(\tilde{a} u^2+\tilde{\alpha} f(x))v-b u + k\Delta u, \\
    \frac{\partial v}{\partial t}&=-(\tilde{a} u^2+\tilde{\alpha} f(x))v+b u + k_v\Delta v.
\end{split}
\end{align}

In \eqref{eqn:original_sys}, $\tilde{a},b,\tilde{\alpha},k,k_v$ are constants. $\tilde{a}$ is the reaction activation constant, $b$ is the reaction depletion constant, $\tilde{\alpha}$ is the overall pheromone strength, $k$ is the diffusion coefficient for $u$, and $k_v$ is the diffusion coefficient for $v$. $f(x)$ is a bounded smooth function that describes the pheromone level at different locations.

We assume that rho-GDPase diffuses infinitely fast, i.e, $k_v$ approaches $\infty$. Since the total mass of rho-GTPase and rho-GDPase is conserved, $M=\int( u(\cdot,t)+v(\cdot,t))dx$ is a constant. Then we can obtain the following equation \eqref{eqn:dim_system} that describes the activator-substrate reaction between these two substances. The setting we have is $x \in \mathbf{T}^d$ when dimension $d=1,2$, with periodic boundary condition:
\begin{equation}\label{eqn:dim_system}
    \frac{\partial u}{\partial t}=(\tilde{a} u^2+\tilde{\alpha} f(x))\frac{1}{|{\mathbf{T}^d}|}\left(M-\int_{{\mathbf{T}^d}} u dx\right)-b u + k\Delta u.
\end{equation}
In \eqref{eqn:dim_system}, $|{\mathbf{T}^d}|$ is the measure of the domain, $M$ is the total mass. We are interested in the non-negative solution $u$ with $\int u dx \leq M$ for all time.
By rescaling space and time, we can simplify the equation \eqref{eqn:dim_system} as follows:
\begin{equation}\label{eqn:full_system}
    \frac{\partial u}{\partial t}=(a u^2+\alpha f(x))\left(M-\int_{{\mathbf{T}^d}} u dx\right)-u + \Delta u.
\end{equation}
Here depletion rate and diffusion coefficient are normalized to $1$, and $\frac{1}{|{\mathbf{T}^d}|}$ gets absorbed into $\tilde{a}$ and $\tilde{\alpha}$.
\begin{figure}[htb!]
    \centering
    \includegraphics[width=0.7\linewidth]{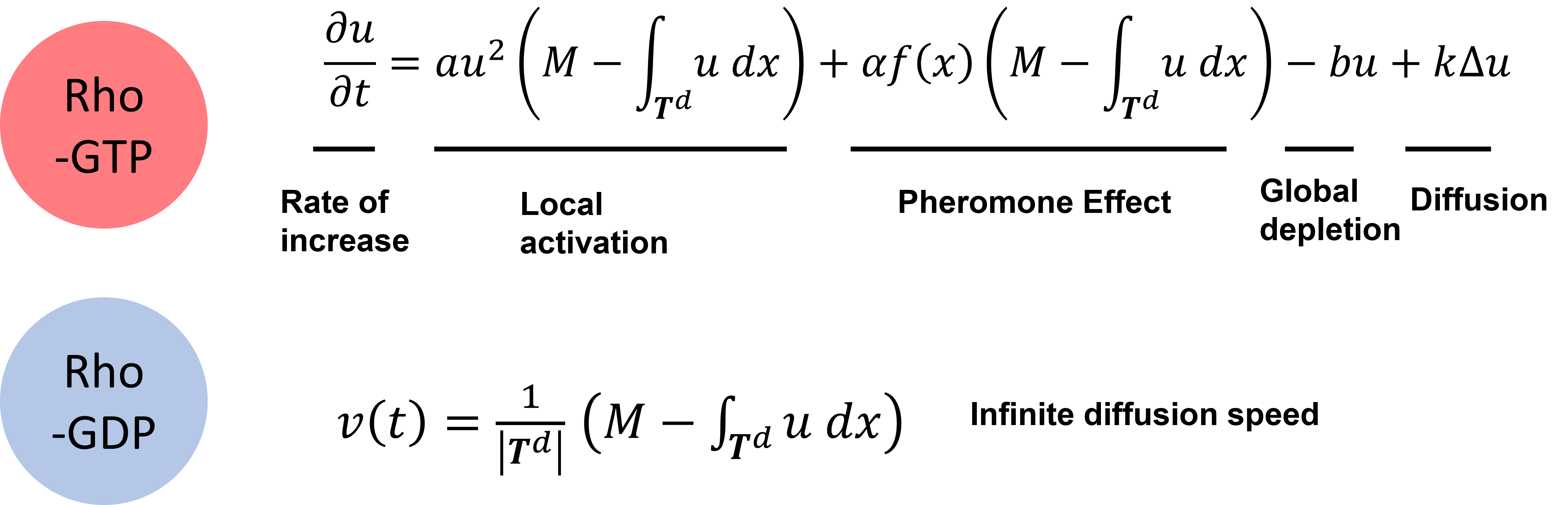}
    \caption{The interconversions of Rho-GTPases between active and inactive forms can be modeled as an reaction-diffusion equation governing the dynamics of the slowly-diffusing activator $u$ and the infinitely-diffusing substrate $v$.}
    \label{fig:phi}
\end{figure}

Our main results are as follows.
On the rigorous level, we are able to establish global regularity results for equation \eqref{eqn:full_system} in one and two dimensions for all non negative initial data.
To better understand the structure of the equation, we consider \eqref{eqn:full_system} in the absence of regularizing diffusion, and prove that for all finite $0\leq t\leq T\leq\infty$, the profile for active rho-GTPase $u$ stays smooth even when there is no diffusion term.
When we re-introduce diffusion in one dimension, uniform in time global bound on derivatives of $u$ is shown. With diffusion in two dimension, global regularity with possible growth is proved. In our numerical experiments, we observe the profile of active rho-GTPase $u$ move towards higher concentration of pheromone, and stops moving once it reaches the location with maximum pheromone concentration. In addition, we explore the speed of such movement through tracking the center of mass of rho-GTPase profile. If pheromone concentration is linear, the center of mass moves with a constant speed towards the pheromone peak. 
More importantly, we verify that the movement speed depends linearly on the pheromone gradient in a natural parameter range similar to that used in \cite{chiou2018principles}.
Note that such linear dependence of chemotactic drift on the gradient of the attractive chemical density $f(x)$ is a common feature of chemotaxis models, including the most studied Keller-Segel equation which in its simplest form reads (see e.g. \cite{perthame2006transport})
\begin{equation}\label{KS} \partial \rho -\Delta \rho + \alpha \nabla \cdot (\rho \nabla f) =0. \end{equation}
The emergence of transport mean field equations such as \eqref{KS} from kinetic equations has been extensively studied (see e.g. \cite{othmer2000diffusion,othmer2002diffusion,james2013chemotaxis,perthame2018flux}). However the existence of chemotactic transport is already built in on the kinetic level.
As far as we know, the equation \eqref{eqn:full_system} is the first simple reaction-diffusion model that aims to rigorously analyze the emergence of chemotaxis from the inner cell workings, even if it is focused on just one stage
of the process that can be quite complex.


The paper is organized as follows: in the next section, we introduce the general set up and key parameters of the model in more detail and present our numerical scheme. We then proceed to describe results of the numerical experiments. After this we state the rigorous results that we are able to prove, and proceed with the proofs.



\section{General Set Up and Numerical Scheme}

We want to explore the origin of chemotactic ability of cells with simulation in $1D$ using the following equation (dropping the scripts in \eqref{eqn:dim_system} and denote the total mass of $u(x,t)$ as $U(t):=\int_{{\mathbf{T}^d}} u(x,t)dx$):
\begin{equation}\label{eqn:full}
    \frac{\partial u}{\partial t}=({a} u^2+{\alpha} f(x))\frac{1}{|\mathbf{T}^d|}\left(M-U(t)\right)-b u + k\Delta u.
\end{equation}

The parameters we use are the same as in \cite{chiou2018principles} and are shown in Table \ref{Table_1} with some basic conversions. We used the method of lines to turn spatially discretized PDE into a system of ODEs, then we use a robust ODE solver ODE15s in Matlab to solve. Note that since we assume rho-GDPase $v$ is rapidly diffusing, we use Simpson's method to numerically integrate rho-GTPase to obtain $\int u(x, t)dx$, and calculate rho-GDPase  as follows:
\begin{equation}
    v(t) = \frac{1}{|\mathbf{T}^d|}\left(M-U(t)\right).
\end{equation}
For the computational part, we restrict ourselves to one dimensional surface and assume the pheromone profile is generated by another yeast cell.
{\color{black} If this cell is some distance away, one  reasonable model of the two-dimensional pheromone distribution is a solution to the heat equation $\partial_t \omega -\Delta \omega =0$ with a $\delta$ function initial data, that is, as a fundamental solution of 2D heat equation.}
Then we can derive the pheromone profile on the cell boundary as shown in Figure \ref{fig:Pheromone_derivation}, and in general, it has the form:
\begin{equation}
    f_h(x) = \frac{\beta}{4\pi \gamma t}\left(\text{exp}\left(-\frac{1}{4 \gamma t} \left( L^2 + r^2 -2Lr\text{cos}\frac{x-x_\text{peak}}{r}\right) \right)\right),
\end{equation}
where $\phi = \frac{x-x_\text{peak}}{r}$ and $\phi\in[-\pi,\pi)$. We use $\gamma$ to denote the diffusion coefficient for the source, $\beta$ as the response strength to the source, and without loss of generality, we assume $t=1$.
\begin{figure}[htb!]
    \centering
    \includegraphics[width=0.6\linewidth]{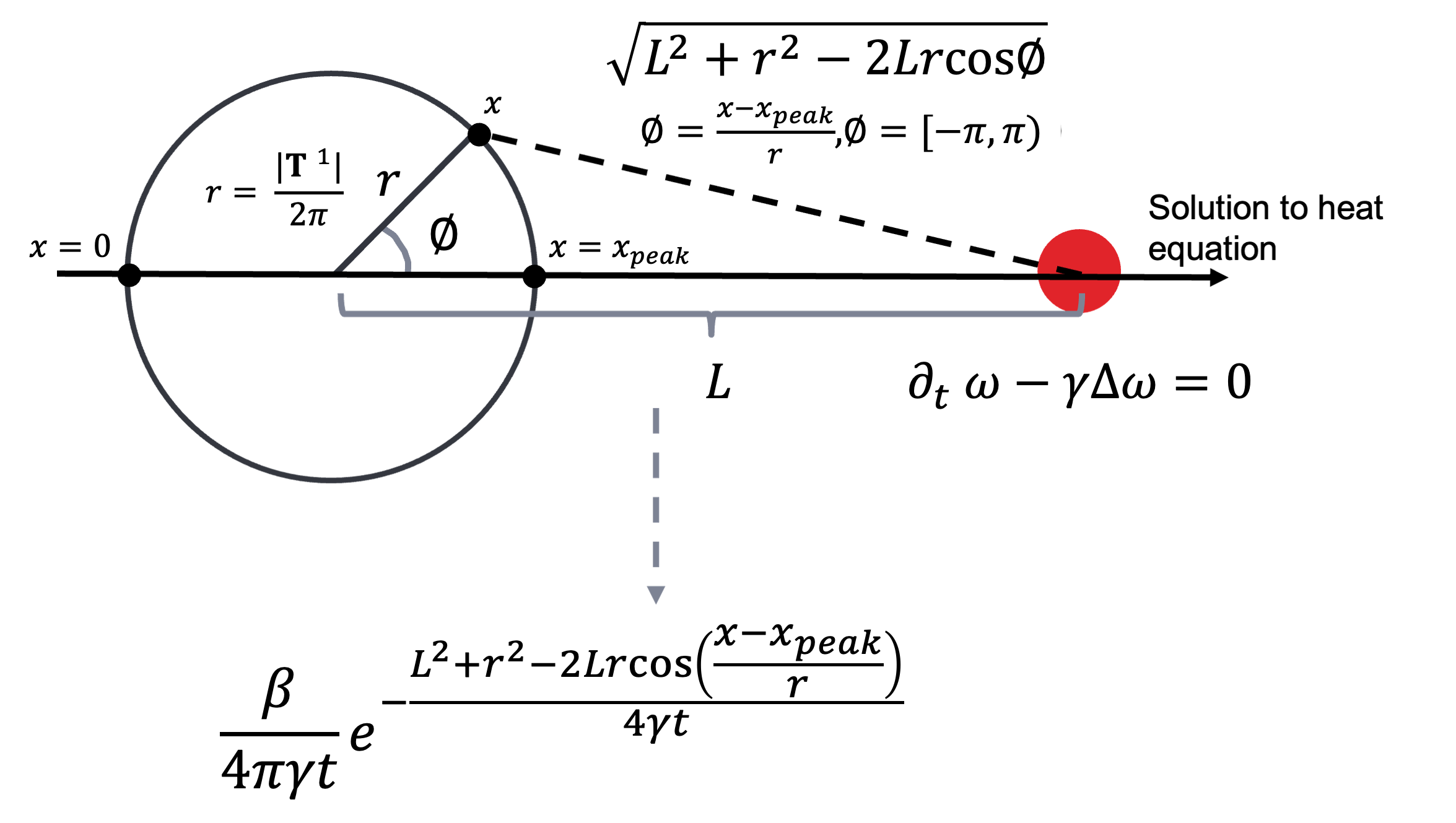}
    \caption{Derivation of a pheromone profile generated from a heat equation.}
    \label{fig:Pheromone_derivation}
\end{figure}
With $\gamma=10$, $t=1$, $L=10$, $\beta=1500$, we can plot $f_h(x)$ in Figure. \ref{fig:Pheromone}. We also plot a linear pheromone profile $f(x)$ that is similar to $f_h(x)$ with a peak at $x_\text{peak}$ defined below in Figure \ref{fig:Pheromone} as well.
{\color{black} The reason for this choice will be explained later.}
\begin{equation}
f(x)=
    \begin{cases}
        \frac{2}{5}x\hspace{1mm} \text{, if } 0\leq x\leq 5, \\
        4- \frac{2}{5}x\hspace{1mm} \text{, if } 5\leq x\leq 10.
    \end{cases}
\end{equation}

\begin{figure}[htb!]
    \centering
    \includegraphics[width=0.6\linewidth]{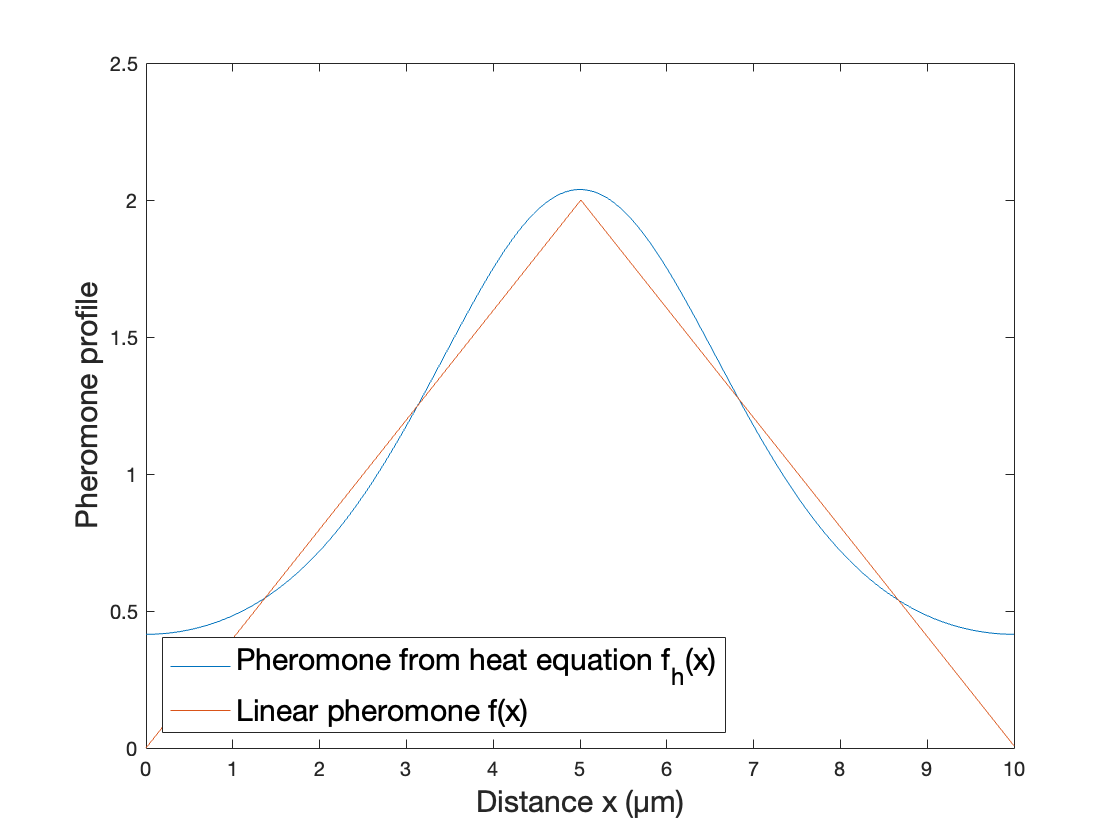}
    \caption{Pheromone profile generated from a heat equation $f_h(x)$ and a similar linear pheromone profile $f(x)$.}
    \label{fig:Pheromone}
\end{figure}

\begin{table}[ht]
\begin{tabular}{llll}
Parameter   & Description &   Value/Range in Simulation  & Unit\\
$k$   & Diffusion coefficient for $u$   & $0.01$  & $(\mu m)^2(s^{-1})$ \\
$\alpha$   & Pheromone strength       & $0.5-3$       & $s^{-1}$ \\
$|\mathbf{T}^d|$    & Cell size    &  10       &     $\mu m$            \\
$M$   & Total mass   & $10$ & -\\
$a$ & Strength of activation & $ 1$ & $(\mu m)^2$\\
$b$ &  Strength of substrate & $ 1$ & $s^{-1}$
\end{tabular}
\caption{Parameters from \cite{chiou2018principles}}\label{Table_1}
\end{table}

To obtain the initial profile of $u$, we start with a uniformly distributed $v$ and a bump function $u$, and we run the simulation without pheromone until it stabilizes according to
\begin{equation}\label{eqn:initial}
    \frac{\partial u}{\partial t}={a} u^2\frac{1}{|{\mathbf{T}^d}|}\left(M-U(t)\right)-b u + k\Delta u.
\end{equation}
Switching on the pheromone, we track the movement using center of mass since the profile of $u$ is relatively stable over time. The center of mass as a function of time is defined using:
\begin{equation}\label{eqn:com}
    \text{CM}_u(t) = \frac{\int x u(x, t)dx}{U(t)}.
\end{equation}
We also measure the movement of the profile of $u$ with the time derivative of the center of mass: $\frac{dCM_u(t)}{dt}$.


\section{Numerical Results: Pheromone Induced Movement}

While there is no explicit transport term in \eqref{eqn:full}, we observe movement of the rho-GTPase $u$ profile over time in response to ``reallocation of resources" to more favorable reaction
regions induced by pheromone $\alpha f_h(x)$ and $\alpha f(x)$ as shown in Figure \ref{fig:different_pheromone}. {\color{black} As our numerical simulations show, at least in certain parameter ranges,
this transport appears quite similar to the Keller-Segel-type transport with speed proportional to the concentration gradient.}

\begin{figure}[htb!]
\centering
\begin{subfigure}{.5\textwidth}
  \centering
  \includegraphics[width=0.95\linewidth]{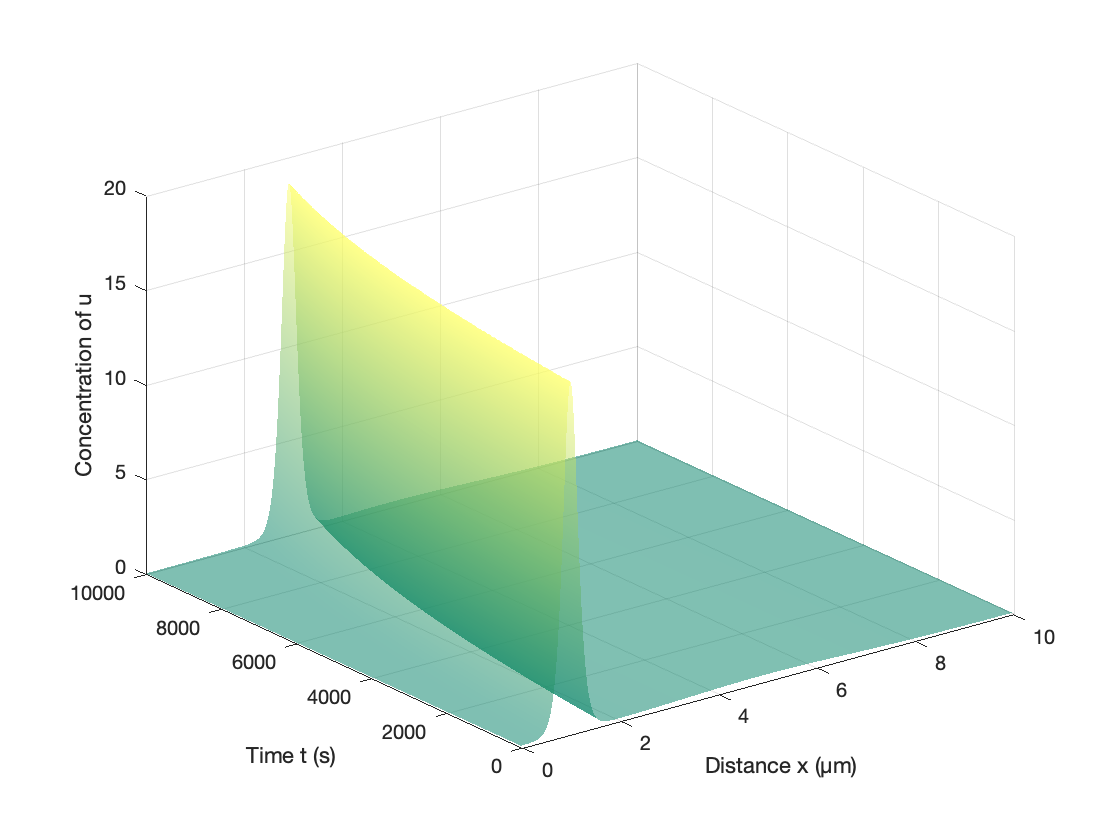}
  \caption{Pheromone profile given by $f_h(x)$.}
  \label{fig:sub1}
\end{subfigure}%
\begin{subfigure}{.5\textwidth}
  \centering
  \includegraphics[width=0.95\linewidth]{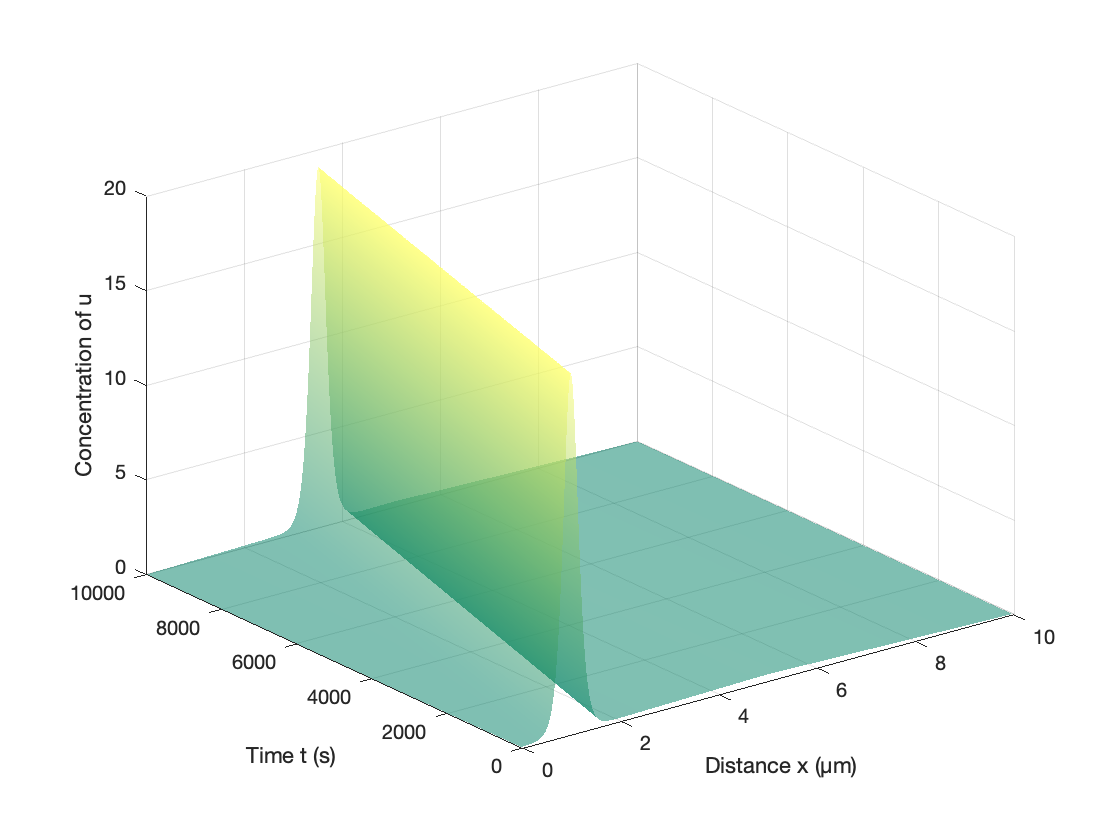}
  \caption{Pheromone profile given by $f(x)$.}
  \label{fig:sub2}
\end{subfigure}
\caption{Numeric solution to equation \eqref{eqn:full} with $\alpha$ = 2 and other parameters given in Table \ref{Table_1}.}
\label{fig:different_pheromone}
\end{figure}

As we can see in Figure \ref{fig:different_pheromone}, one does not expect that the solution will be an exact traveling pulse since the background level of the pheromone affects the shape of the bump of $u$ density.
With time $t$ and recorded $CM_u(t)$, we can compute the movement speed of the center of mass, $\frac{d CM_u(t)}{dt}$ and the corresponding pheromone profile $f_h$ and $f$ at the center of mass. From Figures \ref{fig:nonlinear_relationship} and \ref{fig:linear_relationship}, we propose the hypothesis that the movement speed $\frac{d CM_u(t)}{dt}$ depends on the derivative of the pheromone.
\begin{figure}[htb!]
    \centering
    \includegraphics[width=0.9\linewidth]{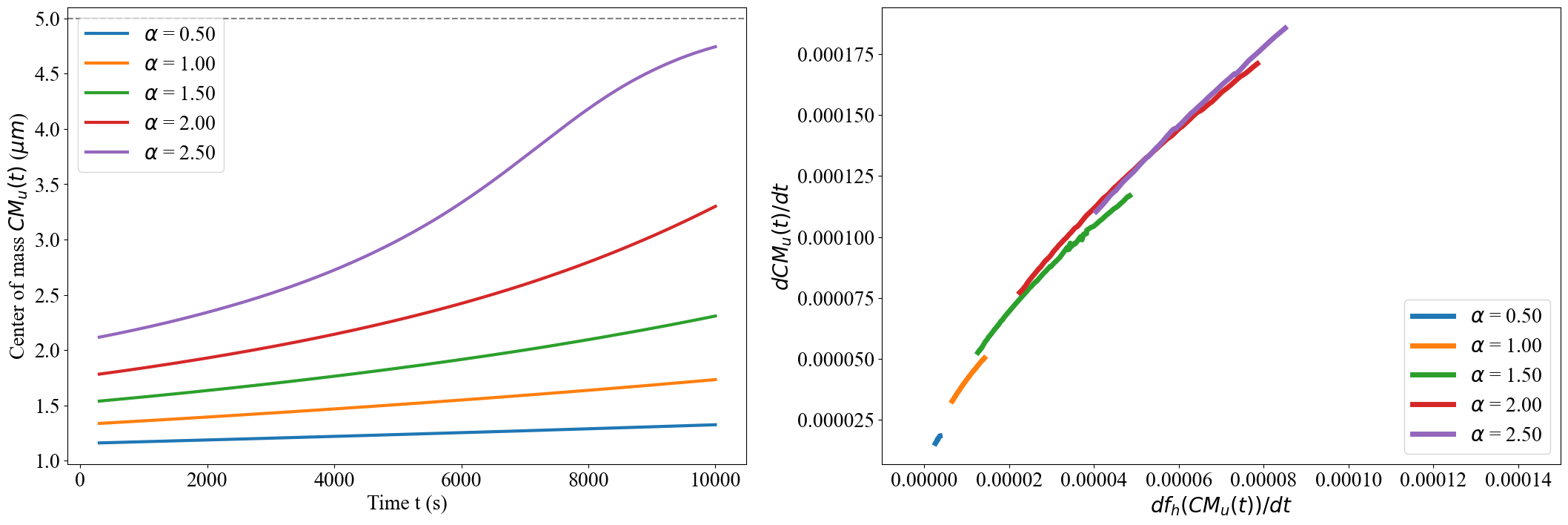}
    \caption{Movement speed $\frac{\text{CM}_u(t)}{dt}$ as a function of pheromone profile slope $\frac{f(_h\text{CM}_u(t))}{dt}$. $\frac{df_h(x)}{dx}>0$ for $x\in(0,2.5).$}
    \label{fig:nonlinear_relationship}
\end{figure}
\begin{figure}[htb!]
    \centering
    \includegraphics[width=0.9\linewidth]{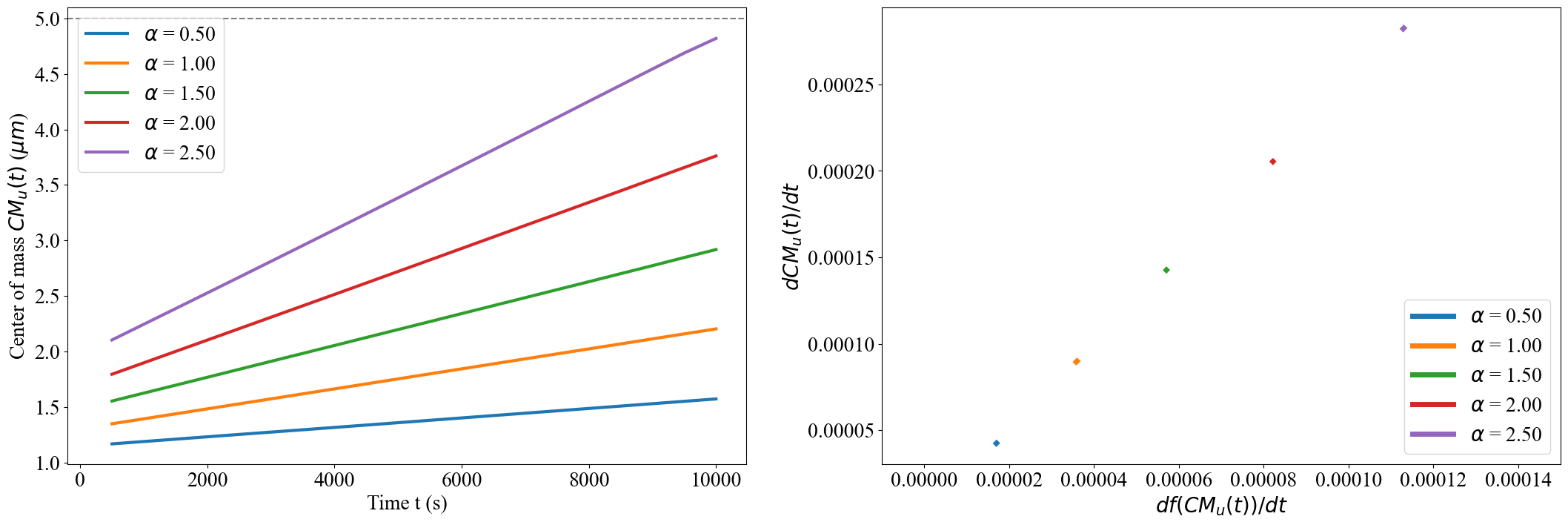}
    \caption{Movement speed $\frac{\text{CM}_u(t)}{dt}$ as a function of pheromone profile slope $\frac{f(\text{CM}_u(t))}{dt}$.  $\frac{df(x)}{dx}$ is a constant for $x\in(0,2.5).$}
    \label{fig:linear_relationship}
\end{figure}

{\color{black} We have tried some other pheromone profiles with similar results. While the graphs on \ref{fig:linear_relationship} appear close to lines, they are not quite lines - but perhaps because the
density bump of $u$ has spatial scale, and so is exposed to a range of concentration slopes (we take the slope at the center of mass as the basis of the functional relationship pictured in this figure).
As one of our goals is to study the dependence of the movement speed of the center of mass $CM_u(t)$ and on the gradient of the pheromone concentration $\alpha f'(x)$, we will also consider the piecewise linear pheromone profile $f(x):$ it has extended regions with
the constant slope that makes it easier to capture the effect more precisely.}
We illustrate this transport picture in Figure \ref{fig:traj2} by  calculating the center of mass using \eqref{eqn:com} of the initial profile of $u$ and the profile at $t=10000s$ pheromone profile $f(x)$ with pheromone strength $\alpha=2$.

As one can expect, the profile of $u$ slows down once its center of mass starts to approach $x=x_{\text{peak}}=5\mu m$.
We can plot the center of mass a function as a time of time for different pheromone strength $\alpha$.
As presented in Figure \ref{fig:com}, the center of mass of $u$ stays at $x=x_{\text{peak}}=5\mu m$ after $t=7000s$ for the pheromone strength $\alpha=3.$
In fact, if we run the simulation long enough, the center of mass of $u$ appears to get arbitrarily close to $x=x_{\text{peak}}=5\mu m$ for all $\alpha >0$.
\begin{figure}[hbt!]\label{Fig:chemotaxis_cutoff}
    \centering
     \hspace{-4em}
    \begin{subfigure}[b]{0.5\textwidth}
        \centering
        \includegraphics[width=0.95\linewidth]{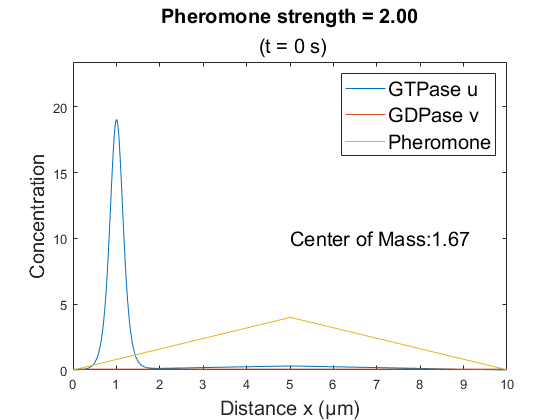}
        \caption{Initial profile $t = 0 s$.}
        \label{Fig:plateau effect}
    \end{subfigure}%
    \hspace{-2em}%
    \begin{subfigure}[b]{0.5\textwidth}
        \centering        \includegraphics[width=0.95\linewidth]{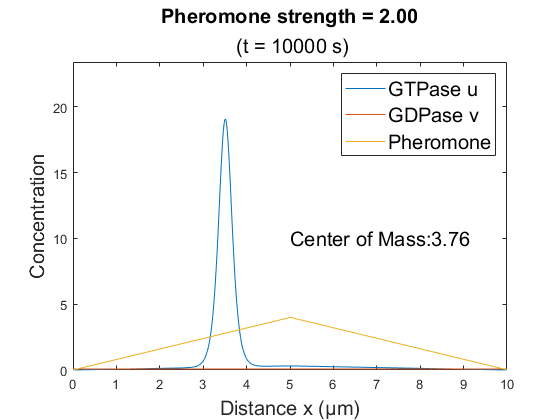}
        \caption{Profile when $t = 10000 s$.}
        \label{Fig:optimal_shear}

    \end{subfigure}
     \hspace{-4em}
    \caption{Initial profile for rho-GTPase $u$, rho-GDPase $v$, and pheromone profile $f(x)$ and pheromone strength $\alpha = 2.$}
    \label{fig:traj2}
\end{figure}

We then continue to explore the movement speed of the center of mass as a function of pheromone strength $\alpha,$ {\color{black} which controls the slope of the pheromone concentration}. From Figure \ref{fig:com}, we can see a constant movement speed of the center of mass when the profile of $u$ is far away from $x_\text{peak}$. Moreover, if we plot the movement speed (before the profile of $u$ is too close to $x_\text{peak}$) as a function of $\alpha$ as shown in Figure \ref{fig:speed}, we observe linear dependence of transport speed on pheromone strength. Such linear dependence corresponds to the mean field chemotaxis models in  \cite{keller1971model,perthame2006transport}.

\begin{figure}[htb!]
    \centering
    \includegraphics[width=0.6\linewidth]{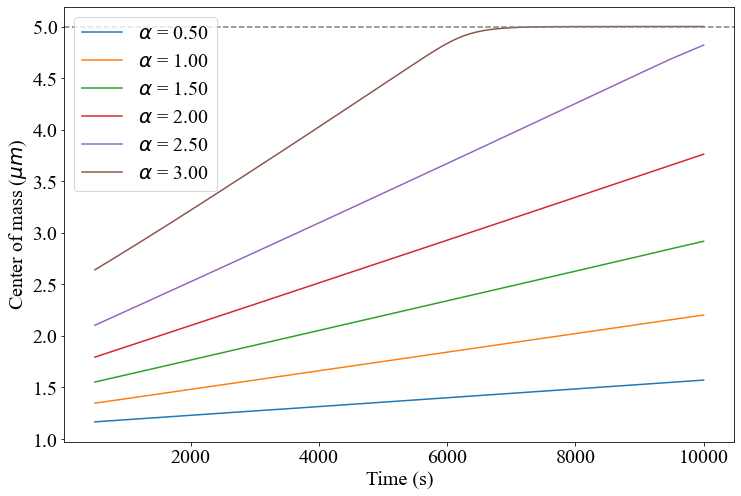}
    \caption{Center of mass position $\text{CM}_u(t)$ as a function of time with pheromone profile $f(x)$. In regions away from $x_\text{peak}$, $\text{CM}_u(t)$ changes linearly with time.}
    \label{fig:com}
\end{figure}

\begin{figure}[htb!]
\includegraphics[width=0.6\linewidth]{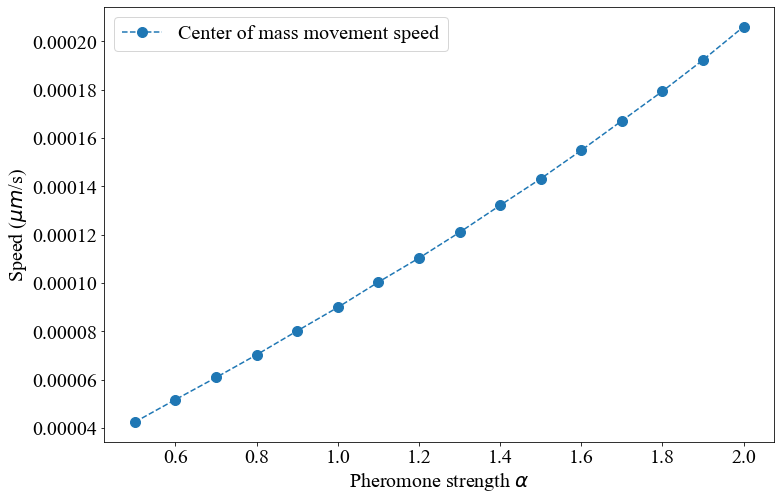}
 \centering
 \caption{Center of mass movement speed $\frac{d\text{CM}_u(t)}{dt}$ as a function of pheromone strength $\alpha$ with pheromone profile $f(x)$. In regions away from $x_\text{peak}$, $\frac{d\text{CM}_u(t)}{dt}$ changes linearly with $\alpha$.}
 \label{fig:speed}
\end{figure}

\section{Mathematical Analysis: Global Regularity}
Our goal in this section is to initiate rigorous mathematical analysis of the equation \eqref{eqn:full_system}. There is much literature on regularity of solutions for reaction-diffusion type equations, but we could not locate references dealing directly with \eqref{eqn:full_system} due to nonlocality produced by our assumption of infinite diffusion for $v.$
In \cite{bian2019global}, the author presented global existence result to a similar equation for sufficiently small and non-negative initial data. In this section, we will present regularity result for \eqref{eqn:full_system} in $\mathbf{T}^d$, with $d=1,2$ for any non-negative smooth initial data with diffusion and without diffusion. When there is no diffusion with $d=1,2$, we can still show that $u$ is smooth for all times $0\leq t< \infty$. With diffusion in 1D, uniform in time global bound is proven while with diffusion in 2D, global regularity with possible growth is shown.
{\color{black} At this time, we are unable to prove rigorous results that would provide more qualitative features of the evolution, and in particular establish the conjecture on linear dependence of the transport speed on pheromone gradient at least in some regimes.
Part of the difficulty is that, as we mentioned before, the solution cannot be expected to take form of a traveling pulse of the fixed shape on an unbounded domain, and lack of such framework complicates analysis.
This paper can be viewed as creating a foundation for such further investigation that remains an intriguing open problem. }

\subsection{Without diffusion}
 We start with the case without diffusion. Consider the following equation:
  \begin{equation}\label{eqn:no_diffusion}
    \frac{\partial u}{\partial t}=(a u^2+\alpha f(x))\left(M-U(t)\right)-u, U(t)=\int_{\mathbf{T}^d} u dx.
\end{equation}
\theorem Suppose $u(t,x)$ is a non-negative solution to \eqref{eqn:no_diffusion} with dimension $d=1,2$ and periodic boundary condition; $a,\alpha,M$ are constant parameters, and $f(x)$ is a smooth non-negative function. If $u_0(x)$ is a smooth initial profile, then $u(t,x)$ stays smooth for all times {\color{black} $0\leq t <\infty.$ }
\begin{proof}
We first show some a-priori bounds on $u$, namely that all Sobolev norms of $u$ are controlled by $L^\infty$ norm. It is clear that on a bounded domain, $L^\infty$ norm controls all other $L^p$ norms. In the estimates that follow, $D$ stands for any partial derivative (just $\partial_x$ in one dimension),
and $W^{k,p}$ is the usual Sobolev space. Multiplying \eqref{eqn:no_diffusion} by $(-\Delta)^s u$ and integrating, we get
\begin{equation}\label{eqn:no_diff_hs}
    \partial_t ||u||_{H^s}^2 \leq C\underbrace{|\int u^2 \cdot (-\Delta)^s udx |}_{I}-\underbrace{\int u (-\Delta)^s u dx}_{||u||_{H^s}^2}+\alpha M |\int f (-\Delta)^s u dx|.
\end{equation}
The last term can by controlled by $\alpha M ||f||_{W^{2s,1}}||u||_{L^\infty}$ (moving all derivatives to $f$). Term $I$ can be represented by a sum of integrals of the type $\int D^l u D^{s-l} u D^s u dx $, where $l=0,\dots, s.$ Then with H{\"o}lder's inequality and Gagliardo–Nirenberg interpolation inequality in dimension $d=1,2$, we can bound them by:
\begin{equation*}
        \int D^l u D^{s-l} u D^s u dx  \leq C ||D^l u||_p ||D^{s-1}u||_q||D^s u||_2 \hspace{1mm}\text{with} \hspace{1mm} \frac{1}{p}+\frac{1}{q}+\frac{1}{2}=1;
\end{equation*}
\begin{equation*}
    ||D^l u||_p \leq C ||D^s u||_{2}^{\alpha}||u||_\infty^{1-\alpha};
\end{equation*}
\begin{equation*}
    ||D^{s-l} u||_q \leq C ||D^s u||_{2}^{\beta}||u||_\infty^{1-\beta}.
\end{equation*}
In $1D$, $\alpha=\frac{2(-1+lp)}{p(-1+2s)}$ and $\beta=\frac{2(-1+(s-l))q}{q(-1+2s)}$, and $\alpha+\beta=2(\frac{1}{p}+\frac{1}{q})=1$.
\newline
In $2D$, $\alpha=\frac{-2+lp}{p(-1+s)}$ and $\beta=\frac{-2+(s-l)q}{q(-1+s)}$, and $\alpha+\beta=2(\frac{1}{p}+\frac{1}{q})=1$.
Then
\begin{equation*}
    \int D^l u D^{s-l} u D^s u dx  \leq C ||D^l u||_p ||D^{s-1}u||_q||D^s u||_2\leq C||u||_{H^s}^2||u||_\infty.
\end{equation*}
Substituting back into \eqref{eqn:no_diff_hs} gives:
\begin{equation*}
    \partial_t ||u||_{H^s}^2 \leq (C||u||_\infty-1)||u||_{H^s}^2+\alpha M ||f||_{W^{2s,1}}||u||_{L^\infty}.
\end{equation*}
Applying Gronwall inequality \cite{gronwall1919note}, we see that to show global regularity, it suffices to prove that
 $\int_0^T ||u(\cdot, t)||_\infty\,dt$ remains bounded. We will show a stronger constraint that $\|u(\cdot, t)\|_\infty$
 remains finite for all times. Via contradiction, denote $T$ the first time of blow up of $||u||_\infty$.
\newline
Consider $\rho=\frac{e^{-t}}{1+u}$, then $\rho$ satisfies the following equation:
\begin{equation}\label{eqn:rho}
    \frac{\partial \rho}{\partial t}=-(a e^{-t}-2a\rho +a\rho^2 e^t+\alpha f(x)\rho^2e^t)\left(M-U(t)\right)-\rho^2e^t.
\end{equation}
Then at time $T$, $\rho$ will reach $0$ at some point. 
For simplicity, let us focus on $d=1$; the argument for $d=2$ is similar. From \eqref{eqn:rho}, we see that $\rho$ is decreasing for all times, therefore $\rho$ is bounded from above by $||\rho_0||_\infty$. Now we take a derivative with respect to $x$ to obtain:
\begin{equation*}
\partial_x\partial_t \rho = -\left(M-U(t)\right)\left(-2a\partial_x \rho+2ae^t \rho\partial_x \rho +\alpha\rho^2 e^t \partial_x f+2\alpha\rho e^t f(x)\partial_x\rho\right) -e^t\rho\partial_x\rho.
\end{equation*}
Then it is clear that for all $x,$
\begin{equation*}
    \partial_t|\partial_x \rho|\leq C_0(T)+C_1(T)|\partial_x \rho|.
\end{equation*}
By Gronwall inequality, we can see that $| \partial_x \rho |$ is bounded for any finite time $t \leq T$.
We can take another derivative in $x$ and obtain:
\begin{equation*}
\begin{split}
\partial_{xx}\partial_t \rho =& -\left(M-U(t)\right)(-2a\partial_{xx} \rho+2ae^t \rho\partial_{xx} \rho +2ae^t(\partial_x \rho)^2+\alpha\rho^2  e^t \partial_{xx} f + \alpha\rho e^t \partial_x f \partial_x\rho \\
&+2\alpha\rho e^t f(x)\partial_{xx}\rho+2\alpha e^t f(x)(\partial_x\rho)^2+2\alpha\rho e^t \partial_x f\partial_x\rho) -e^t\rho\partial_{xx}\rho-e^t(\partial_x\rho)^2.
\end{split}
\end{equation*}
Boundedness of $\rho$ along with control of $|\partial_{x}\rho|$ give:
\begin{equation*}
    \partial_t|\partial_{xx}^2 \rho|\leq C_2(T)+C_3(T)|\partial_{xx} \rho|.
\end{equation*}
By Gronwall inequality again, one gets that $|\partial_{xx}^2\rho|$ is bounded for any finite time $t \leq T$.
One can effectively continue this calculation and get that all derivatives in space are bounded for $t\leq T$.
Since blow up happens for the first time at time $T$, then $\rho(x_B,T)=0$ at some point $x_B$. There can be many such points, but let us focus on one of them. Due to control of $\partial_{x}|\rho|$ and $\partial_{xx}^2|\rho|$, we have $\rho(x,T)\leq C(T)|x-x_B|^2$ by Taylor expansion.  Observe that $\rho(x,t)\rightarrow\rho(x,T)$ monotonically for every $x$. Therefore, as $u=\frac{1}{e^t \rho}-1$, we have $u(x,t)\rightarrow u(x,T)$ (including when $u(x,t)=\infty$). Then we have $u(x,T)=\frac{1}{e^T \rho} -1 \geq \frac{1}{e^T |x-x_B|^2}-1$.
Then by Fatou's lemma, we have
\begin{equation}
    M\geq\liminf_{t->T}\int u(x,t)dx\geq  \int\liminf_{t->T} u(x,t)dx=\int u(x,T)dx \geq \int C|x-x_B|^{-2}dx=\infty,
\end{equation}
which is a contradiction. Therefore, we cannot have such finite time blow up. Note that the argument above works both in 1D and 2D, only the computation yielding control of the derivatives of $\rho$ needs a minor adjustment.
{\color{black}  Note that the size of the Sobolev norms of the solution may depend on time.}
\end{proof}

\normalfont
\subsection{With diffusion in 1D}
Now we turn out attention to the system with diffusion in $1D$: \begin{equation}\label{eqn:diff_1d}
\frac{\partial u}{\partial t}=(au^2+\alpha f(x))\left(M-U(t)\right)-u + \frac{\partial ^2 u}{\partial x^2}, U(t)=\int_{\mathbf{T}^d} u dx.
\end{equation}
We will prove global regularity for \eqref{eqn:diff_1d} as well.

\begin{theorem} Suppose $u$ is a non-negative solution to \eqref{eqn:diff_1d} in dimension $d=1$ and periodic boundary condition. Let $a,\alpha,M$ be constant parameters, and $f(x)$ a smooth function. If $u_0(x)$ is a smooth initial profile, then $u(x,t)$ stays smooth for all time;
in particular, all Sobolev norms $||u||_{H^s}$ with $s>0$ are bounded uniformly for all time.
\end{theorem}
\begin{proof}
First we show that $||u||_2$ is bounded:
multiplying both sides by $u$ and integrating, we have
\begin{equation*}
 \int u \frac{\partial u}{\partial t}dx=\int u(au^2+ f(x))\left(M-U(t)\right)dx-\int u^2dx +\int u  \frac{\partial ^2 u}{\partial x^2}dx.
\end{equation*}
Therefore,
\begin{equation}\label{eqn:diff_bound}
    \frac{1}{2}\partial_t \int u^2dx\leq M \int a u^3 dx -\int u^2dx- \int (u_x)^2dx + M^2||f||_\infty.
\end{equation}
Using Gagliardo-Nirenberg-Sobolev inequality (see eg.\cite{doering1995applied}) gives:
\begin{equation}\label{eqn:u_3}
    ||u||_{L^3}\leq C ||u||_{L^1}^{5/9}||u_x||_{L^2}^{4/9}.
\end{equation}
Substituting \eqref{eqn:u_3} into \eqref{eqn:diff_bound} yields (note that the constant C changes from line to line and may depend on $M$):
\begin{equation}\label{eqn:1d_l2}
\begin{split}
        \frac{1}{2}\partial_t \int u^2dx & \leq C\left(\int u dx\right)^{5/3}\left(\int u_x^2 dx\right)^{2/3} -\int u^2dx- \int u_x^2dx +M^2||f||_\infty\\
        & \leq  \frac{C}{3}\left(\int u dx\right)^{5} + \frac{2}{3} \left(\int u_x^2 dx\right) -\int u^2dx- \int u_x^2dx +M^2||f||_\infty \\
        & \leq -\frac{1}{3} \left(\int u_x^2 dx\right)- \int u^2dx +M^2||f||_\infty+C (M).
\end{split}
\end{equation}
Note that in the second inequality above, we used Young's inequality $ab\leq \frac{a^p}{p}+\frac{b^q}{q}$ with $p=3,q=\frac{3}{2}$. The calculation
\eqref{eqn:1d_l2} implies that $||u||_{L^2}$ is globally bounded by Gronwall inequality.
Moreover, from \eqref{eqn:1d_l2}, we can see that $||u||_{2}$ is in fact uniformly bounded by $M^2||f||_\infty+C(M)$ since if $||u||_2$ ever crosses this value for the first time at $t_0$, $\partial_t||u||^2$ becomes negative, which implies that before $t_0$, $||u||_2$ lies above $M^2||f||_\infty+C M^5$ already. We arrive at a contradiction. Therefore, $||u||_2$ is uniformly bounded for all time.
\\
Next, we estimate the higher order Sobolev norms. In fact, in 1D, this can be done using just control of the $L_1$ norm, but we are going to use $L_2$ norm for convenience as we have shown it remains bounded.
Multiplying \eqref{eqn:full_system} by $(-\Delta)^s u$ and integrating, we get
\begin{equation}
    \partial_t ||u||_{H^s}^2 \leq C\underbrace{\left|\int u^2 \cdot (-\Delta)^s udx \right|}_{I}-\underbrace{\int u (-\Delta)^s u dx}_{||u||_{H^s}^2} -||u||_{H^{s+1}}^2+\alpha M \left|\int f (-\Delta)^s u dx\right|.
\end{equation}
The last term can by controlled by $\alpha M ||f||_{H^{2s}} ||u||_{L^2}$ (moving all derivatives to $f$). Term $I$ can be represented by a sum of integrals of the type $\int D^l u D^{s-l} u D^s u dx $, where $l=0,\dots, s.$ In 1D, $D=\partial_x$. Then with H{\"o}lder's inequality and Gagliardo–Nirenberg interpolation inequality with dimension $d=1$, we can bound them by:
\begin{equation}
\begin{split}
        \int D^l u D^{s-l} u D^s u dx & \leq C ||D^l u||_p ||D^{s-1}u||_q||D^s u||_2 \hspace{1mm}\text{with} \hspace{1mm} \frac{1}{p}+\frac{1}{q}+\frac{1}{2}=1\\
        &\leq C||u||_{H^{s+1}}^{2-\frac{3}{2(s+1)}}||u||_{L^2}^{1+\frac{3}{2(s+1)}}.
\end{split}
\end{equation}
Here we deploy the Gagliardo Nirenberg inequalities $||D^su||_{L^2}=||u||_{H^s}\leq ||u||_{L^2}^{\frac{1}{s+1}} ||u||_{H^{s+1}}^{\frac{s}{s+
1}}$, $||D^l u||_p\leq C ||u||_{L^2}^{1-\alpha}||D^{s+1}u||_{L^2}^{\alpha}$, with $\alpha =\frac{l-\frac{1}{p}+\frac{1}{2}}{s+1}$, and $||D^{s-l}u||_q\leq C||u||_{L^2}^{1-\beta}||D^{s+1}u||_{L^2}^\beta$, with $\beta=\frac{s-l -\frac{1}{q}+\frac{1}{2}}{s+1}$.
Substituting gives (note that the constant $C$ changes from  line to line):
\begin{equation}
\begin{split}
     \partial_t ||u||_{H^s}^2 &\leq C||u||_{H^{s+1}}^{2-\frac
     {3}{2(s+1)}}||u||_{L^2}^{1+\frac{3}{2(s+1)}}-||u||_{H^s}^2-||u||_{H^{s+1}}^2+\alpha M ||f||_{H^{2s}} ||u||_{L^2}\\
     &\leq \frac{4s+1}{4s+4}||u||_{H^{s+1}}^2-||u||_{H^{s+1}}^2+\frac{3 C }{4s+4} ||u||_{L^2}^{\frac{4s+10}{3}}-||u||_{H^s}^2+\alpha M ||f||_{H^{2s}} ||u||_{L^2}\\ &\leq -||u||_{H^s}^2+ \frac{3 C }{4s+4} ||u||_{L^2}^{\frac{4s+10}{3}}+\alpha M ||f||_{H^{2s}} ||u||_{L^2},\label{aux1111}
     \end{split}
\end{equation}
where we use Young's inequality $ab\leq\frac{a^p}{p}+\frac{b^q}{q}$ with $p=\frac{4s+4}{4s+1}$, and $q=\frac{4s+4}{3}$.
Given that we proved $\|u\|_2$ is bounded uniformly in time, \eqref{aux1111} implies that
 $||u||_{H^s}$ is also bounded uniformly for all time.
\end{proof}

\subsection{With Diffusion in 2D}
The equation that we are interested in is given by:
\begin{equation}\label{eqn:diff_2d}
    \partial _t u =a u^2 (M-U(t))+\alpha f(x)(M-U(t))-u+\Delta u, U(t)=\int_{\mathbf{T}^d} u dx.
\end{equation}
\begin{theorem}
Suppose $u$ is a non-negative solution to \eqref{eqn:diff_2d} with dimension $d=2$ and periodic boundary condition. Let $a,\alpha,M$ be constant parameters, and $f(x)$ a smooth function. If $u_0(x)$ is a smooth initial profile, then $u(x,t)$ stays smooth for any finite time, that is, Sobolev norms $||u||_{H^s}$ with $s>0$ are bounded for any time {\color{black} $0 \leq t <\infty$}. The bound on the Sobolev norms may now depend on time.
\end{theorem}
\begin{proof}
First we derive an a-priori estimate. In the 2D case, the analog of the estimate \eqref{eqn:1d_l2} is not available, as the exponents do not
allow to control $\|u\|_{L^2}$ uniformly in time in this way. Therefore, we need a more nuanced argument.
Note that
\begin{equation}
   \int_0^T\int  \partial_t u dx dt = \int_0^T\int  \left(a u^2 (M-U(t))+\alpha f(x)(M-U(t))-u+\Delta u\right) dx dt
\end{equation}
gives (in the following calculations, the constant $C$ changes from line to line):
\begin{equation}\label{eqn:timebound_L2}
\int_0^T\int  \left(a u^2 (M-U(t)) \right )dx dt \leq U(T)+ MT.
\end{equation}
Then multiplying \eqref{eqn:diff_2d} by $(-\Delta)^s u$ and integrating in $x$, we obtain
\begin{equation}\label{eqn:global_2D}
    \partial_t ||u||_{H^s}^2\leq -||u||_{H^{s+1}}^2 -||u||^2_{H^s}+a(M-\int u dx)\underbrace{\int u^2(-\Delta)^s u dx}_{I}+\alpha M ||f||_{H^s} ||u||_{H^s}.
\end{equation}
The integral $I$ is a sum of terms of the form: $\int D^l u D^{s-l} u D^s u dx $, and $l=0,\dots, s.$ We estimate it as follows:
\begin{equation}
\begin{split}
    \int D^l u D^{s-l} u D^s u dx & \leq C ||u||_{H^s} ||D^l u||_p ||D^{s-l}||_q \hspace{1mm}\text{with}\hspace{1mm} \frac{1}{p}+\frac{1}{q}+\frac{1}{2}=1\\
    &\leq C ||u||_{H^s}||u||_{H^{s+1}}^{\alpha}||u||_{L^2}^{1-\alpha}||u||_{H^{s+1}}^{\beta}||u||_{L^2}^{1-\beta}\\
    &\leq C||u||_{L^2}||u||_{H^s}||u||_{H^{s+1}} .\\
\end{split}
\end{equation}
Here we use Gagliardo Nirenberg inequalities for $n=2$: $||D^l u||_p\leq C |u||_{H^{s+1}}^{\alpha}||u||_{L^2}^{1-\alpha}$ with $\alpha=\frac{l+1-\frac{2}{p}}{s+1}$, and $||D^{s-l}||_q \leq C||u||_{H^{s+1}}^{\beta}||u||_{L^2}^{1-\beta}$ with $\beta=\frac{s-l+1-\frac{2}{q}}{s+1}$, $\alpha+\beta=1$.
Substituting these estimates back into \eqref{eqn:global_2D} gives:
\begin{equation}
\begin{split}
\partial_t||u||_{H^s}^2 & \leq -||u||_{H^{s+1}}^2 -||u||^2_{H^s}+C a\left(M-U(t)\right)||u||_{L^2}||u||_{H^s}||u||_{H^{s+1}}+\alpha M ||f||_{H^s} ||u||_{H^s} \\
& \leq CMa^2\left(M-U(t)\right)||u||^2_{L^2}||u||^2_{H^s}+\frac{1}{2}||u||^2_{H^{s+1}}-||u||_{H^{s+1}}^2 -||u||^2_{H^s}+\alpha M ||f||_{H^s} ||u||_{H^s}\\
& \leq CMa^2\left(M-U(t)\right)||u||^2_{L^2}||u||^2_{H^s}-\frac{1}{2}||u||^2_{H^{s+1}}+\alpha M ||f||_{H^s} ||u||_{H^s},
\end{split}
\end{equation}
where in the second line, we used the inequality: $2ab\leq \epsilon a^2+\frac{1}{\epsilon}b^2$. Then by Gronwall inequality, we get
\begin{equation*}
||u||_{H^s} \leq \exp\left(CMa\int_0^T\int  au^2\left(M-U(t)\right) dxdt\right) \left(||u(0)||_{H^s}  +\int_0^T \alpha M ||f||_{H^s} dt\right).
\end{equation*}
From \eqref{eqn:timebound_L2}, we see that $||u||_{H^s}$ is bounded for any finite $t\leq T <\infty$.
\end{proof}

{\bf Acknowledgement}. \rm AK has been partially supported by the NSF-DMS award 2006372. The authors thank Daniel Lew for patiently teaching them some of the relevant biology (all inadequacies are our fault) and helpful discussions.
We are grateful to anonymous referees for useful comments that helped improve the manuscript.

{\bf Disclosure of interest}. \rm The authors report no conflict of interest.

\end{document}